\documentclass[11pt]{article}
\usepackage{amsmath}
 \usepackage{amscd}
 \usepackage{amssymb}
 \usepackage{theorem}

\usepackage{color}

 \usepackage{graphicx}
 \usepackage{mathptmx}       
\usepackage{amsmath}
 \usepackage{amscd}
\usepackage{amssymb}
  \usepackage[all]{xy}
\usepackage{graphicx,xcolor}

\newtheorem{TEO}{Theorem}[section]
\newtheorem{PROP}[TEO]{Proposition}
\newtheorem{LEM}[TEO]{Lemma}
\newtheorem{DEF}[TEO]{Definition}

\theorembodyfont{\normalfont}
\newtheorem{EX}[TEO]{Example}

\newtheorem{REM}[TEO]{Remark}
\theorembodyfont{\normalfont}

\newcommand\Oh{{\mathcal O}}

\newcommand\sF{{\mathcal F}}

\newcommand\sI{{\mathcal I}}
\newcommand\sJ{{\mathcal J}}
\newcommand\sK{{\mathcal K}}

\newcommand\sM{{\mathcal M}}

\newcommand\om{\omega}

\newcommand\fie{\varphi}

\newcommand\dual{\mathrel{\raise3pt\hbox{$\underline{\mathrm{\thinspace d
\thinspace}}$}}}

\newcommand\iso{\cong}

\newcommand\into{\hookrightarrow}
\newcommand\onto{\twoheadrightarrow}

\newcommand\proj{\mathbb P}
\newcommand\Ka{\mathbb K}

\newcommand {\p}{\mathbb{P}}
\newcommand {\red}{\operatorname{red}}

\newcommand\length{\operatorname{length}}

\newcommand\im{\operatorname{Im}}

\newcommand\Ann{\operatorname{Ann}}

\renewcommand\div{\operatorname{div}}
\newcommand\Ext{\operatorname{Ext}}
\newcommand\sExt{\operatorname{{\mathcal E}{\it xt}}}
\newcommand\Hom{\operatorname{Hom}}

\newcommand\sHom{\operatorname{{\mathcal H}{\it om}}}

\newcommand\Sing{\operatorname{Sing}}

\newcommand\Sym{\operatorname{Sym}}
\newcommand\Tor{\operatorname{Tor}}

\newenvironment{proof}[1][]{\noindent\textbf{Proof#1}.  }{{\hfill $\blacksquare$}}

\begin{document}

\title{The canonical ring  of a 3-connected curve
\thanks{This research was partially supported  by Italian MIUR through PRIN 2008 project  ``Geometria delle variet\`a algebriche e dei loro spazi di moduli". }}
\author{Marco Franciosi, Elisa Tenni}
\date{}

\maketitle

\begin{abstract}
Let $C$ be a projective curve either reduced with planar singularities or contained in a smooth  algebraic surface.

 We show that  the canonical ring
$R(C, \omega_C)= \bigoplus_{k\geq 0} H^0(C, {\omega_C}^{\otimes k})$ is generated in degree 1 if
$C$ is 3-connected and not (honestly) hyperelliptic; we show moreover that
 $R(C, L)=\bigoplus_{k\geq 0} H^0(C,L^{\otimes k})$ is generated in degree 1 if $C$ is reduced with planar singularities  and $L$
is an invertible sheaf such that $ \deg L_{|B} \geq  2p_a(B)+1$  for every $B\subseteq C$.

\hfill\break	
{\bf keyword:} algebraic curve,  Noether's theorem,  canonical ring

\hfill\break  {\bf Mathematics Subject Classification (2010)} 14H20,  14C20, 14H51
\end{abstract}

\section{Introduction}

Let $C$ be a projective curve either reduced with planar singularities or contained in a smooth  algebraic surface, $\om_C$ be its   dualizing  sheaf of $C$ and $L$ be an invertible sheaf on~$C$.

  The main result of this paper is Theorem \ref{noether} stating that  the  canonical ring
 $$ \displaystyle{ R(C, \omega_C)=\bigoplus_{k\geq 0} H^0(C, {\omega_C}^{\otimes k})}$$
 is generated in degree 1 if
$C$ is a 3-connected and not honestly  hyperelliptic curve (see Definition \ref{hyperelliptic} and Definition \ref{m-connected}). This  is a generalization to singular curves of  the classical
Theorem of Noether  for smooth curves  (see \cite[\S III.2]{ACGH})  and can be regarded as a first step in a more general analysis of the Koszul groups
$ \sK_{p,q}(C,\omega_C)$  of  3-connected curves (see \cite{Gr} for the definition and the statement of the so called ``{\em Green's conjecture}"). A detailed explanation of the role that the Koszul groups of smooth and singular curves play in the geometry of various moduli spaces can be found in  \cite{Ap-Far}.


Additional motivation for the present work comes from the theory of surface  fibrations,
as shown by  Catanese and Ciliberto in \cite{Ca-Ci} and Reid in  \cite{Reid}.
 Indeed,
given a  surface fibration $ f\colon S\rightarrow B$  over a smooth curve $B$,
 the  {\em  relative
canonical algebra}   $ R (f)= \bigoplus_{n\geq 0} f_{\ast}(\omega_{S/B}^{\otimes n})$ gives important information on the geometry of the surface. It is clear that the behaviour of $R(f)$ depends on the canonical ring of every fibre.

The main result on  the canonical ring   for  singular curves in the literature is the 1-2-3 conjecture,  stated by Reid in \cite{Reid} and proved
   in \cite{fr_adj}  and  \cite {Konno123}, which says that the canonical ring $R(C, \omega_C)$  of a connected Gorenstein curve of arithmetic genus $p_a(C)\geq 3$  is generated in degree
   1, 2, 3, with the exception of a small number of cases.
     More recently  in \cite{fr_pari} the first author proved that the canonical ring is generated in degree 1 under the strong assumption that $C$ is {\em even}
   (i.e., $\deg_B K_C $ is even on every subcurve $B\subseteq C$).

We remark that our result implies in particular that the canonical ring of a  regular surface of general type is generated in degree $\leq 3$ if there exists a curve $C\in|K_S|$ 
3-connected and not honestly hyperelliptic (see \cite[Thm. 1.2]{fr_pari}). Moreover one can applying the same argument of Konno  (see \cite[Thm. III]{Konno123}) and  see that  the
relative canonical algebra of a relatively minimal surface fibration  is generated in degree 1 if every fibre is 3-connected and not honestly hyperelliptic.  \\

Our second result is Theorem \ref{castelnuovo} stating that the ring
$$R(C, L)=\bigoplus_{k\geq 0} H^0(C,L^{\otimes k})$$
considered as an algebra over $H^0(C, \Oh_C)$,     is generated in degree 1 if $C$  is reduced with planar singularities  and $L$
is an invertible sheaf such that $ \deg L_{|B} \geq  2p_a(B)+1$  for every $B\subseteq C$.
This is a generalization  of a Theorem of Castelnuovo (see \cite{Mu}) on the projective normality of smooth  projective curves. 
This result  can be  useful when dealing with properties (for instance Brill-Noether properties) of a  family of smooth   curves  which degenerates to  $C$. 
 \\

In \cite{fr_adj}, \cite{fr_pari}, \cite{Konno123} the analysis of the canonical ring  is  based on the study of the  Koszul  groups $ \sK_{p,q}(C,\omega_C)$  (with $p,q$ small)   and their vanishing  properties,
together with some vanishing results for invertible sheaves of low degree.

In this paper we use a completely  different approach. Our method is  inspired by the arguments developed in  a series of papers by Green and Lazarsfeld which appeared in the late '80s (see
\cite{Gr},  \cite{green_laz})  and it is based on the generalization to singular curves of Clifford's Theorem   given by the authors in \cite{cliff}.

Given an invertible sheaf $L$
such that the map  $H^0(C, L) \otimes H^0(C,L) \to  H^0(C,L^{\otimes 2})$ fails to be surjective, we exhibit
a  $0$-dimensional scheme $S$ such that
 the   map
$H^0(C, L) \otimes H^0(C,L) \to H^0(S, \Oh_S) $ induced by the restriction also fails to be surjective. Thus  its  dual map
$$\varphi \colon \Ext^1(\Oh_S, \omega_C \otimes L^{-1}) \to \Hom(H^0(C, L), H^1(C, \omega_C \otimes L^{-1}))$$
is not injective. From the analysis of an extension in the Kernel of $\varphi$  we  conclude that the cohomology of
$\sI_S \cdot L$ must satisfy some numerical conditions. This in turn contradicts  Clifford's Theorem  when
$L=K_C$, or $\deg L_{|B} \geq 2p_a(B)+1$ on every~$B\subseteq C$.  \\

Finally, we wish to stress the role of numerical connectedness in  generalizing Noether's theorem. By the results of \cite{CFHR} (see \cite[\S 2, \S 3]{CFHR} or Theorem \ref{thm:curve}) and our main result Theorem \ref{noether}
  we have the following implications for a connected curve~$C$
$$\begin{matrix} C \mbox{ 3-connected, not  honestly hyperelliptic} \Longrightarrow\\   R(C, \omega_C)  \mbox{ is generated in degree } 1 \mbox{ and } \omega_C  \mbox{ ample}
\Longrightarrow \omega_C  \mbox{ \ is very ample } \end{matrix}$$
If $C$ is reduced
it is known that the three properties are equivalent (see \cite{CF}).
 However this is false when  $C$ is not reduced.  To see that the converse of the first implication fails one can take $C = 2F$, where $F$ is a non hyperelliptic fibre of a surface fibration.
In Example \ref{esempio} we will construct a curve  with very ample canonical sheaf which fails Noether's Theorem, thus proving that the converse of the second implication is false too.
This examples support our belief  that 3-connected curves are the most  natural  generalization of smooth curves when dealing with the properties of the canonical embedding.

\hfill\break
{\bf Acknowledgments.} 
 The second  author  wishes to thank  the Department of Mathematic 
 of the  University of Pisa, especially Rita Pardini,  for providing an excellent research environment.

\section{Notation and preliminary results}

 We work over  an algebraically closed field $\Ka$ of characteristic $\geq 0$.
 
Throughout this paper a curve $C$ will be a Cohen-Macaulay scheme of pure dimension.  It will be \textit{projective ,  either reduced  with planar singularities}
 (i.e. such that for every point $P\in C$ it is $ \dim_{\Ka} \sM/\sM^2 \leq 2$ where $\sM$ is the maximal ideal of $\Oh_{C,P}$)
\textit{or contained in a smooth algebraic surface}  $X$, in which case we allow $C$ to be reducible and non reduced. Notice that $C$ is Gorenstein.

In both cases we will use the standard notation for curves lying on smooth algebraic surface, writing $C=\sum_{i=1}^{s} n_{i} \Gamma_{i}$, where
 $\Gamma_{i}$ are    the
 irreducible components of $C$  and $n_{i}$ are their multiplicities.

\textit{A subcurve $B\subseteq C$ is a Cohen-Macaulay subscheme of pure dimension 1};  it
will be written as
 $\sum m_{i} \Gamma_{i}$, with $0\leq m_i\leq n_i$ for every $i$.


Notice that under these assumptions every subcurve $B \subset C$ is Gorenstein too. 

 $\om_C$ denotes the   dualizing sheaf of $C$ (see \cite{hart}, Chap.~III, \S7), and
$p_a(C)$ the arithmetic genus of $C$, $p_a(C)=1-\chi(\Oh_C)$.  $K_C$ denotes the canonical divisor.


\begin{DEF}\label{hyperelliptic}
A  curve
 $C$ is {\em honestly
hyperelliptic}  if there exists a finite
morphism \mbox{$\psi\colon C\to\proj^1$} of degree $2$.
(see    \cite[\S3]{CFHR} for a detailed treatment).
\end{DEF}

 If $A,\, B$ are subcurves of $C$ such that $A+B=C$, then their product $A \cdot B$ is
$$A \cdot B =
 \deg_B(K_C)- (2p_a(B)-2) =  \deg_A(K_C)- (2p_a(A)-2). $$
 If $C$ is contained in a smooth algebraic surface  $X$ this corresponds to the intersection product of curves as divisors on $X$.

\begin{DEF}\label{m-connected}
$C$ is {\em m-connected } if for every decomposition $C=A+B$ in effective, both nonzero curves, one has $A \cdot B\geq m$. $C$ is {\em numerically connected} if it is 1-connected.
\end{DEF}

First  we recall some useful
results proved  in  \cite{CF} and
 \cite{CFHR}.
  \begin{TEO}[ \cite{CFHR} \S 2, \S 3]\label{thm:curve}Let $C$ be a
Gorenstein curve. Then
  \begin{enumerate}
	 \renewcommand\labelenumi{(\roman{enumi})}
	
 \item If $C$ is 1-connected then $H^{1}(C,K_C) \iso \Ka$.

 \item If $C$ is 2-connected and $C\not \iso \proj^{1}$ then $|K_{C}|$ is base point free.

\item If $C$ is 3-connected and $C$ is not honestly hyperelliptic (i.e.,
there does not exist a finite
morphism $\psi\colon C\to\proj^1$ of degree $2$) then $K_{C}$ is very ample.
 \end{enumerate}
 \end{TEO}

 \begin{PROP}[ \cite{CFHR}, Lemma 2.4 ]\label{lem:adj} Let $C$ be a
 projective scheme of pure dimension 1, let $\sF$ be a coherent
sheaf on $C$, and $\fie\colon\sF\to\om_C$ a nonvanishing map of $\Oh_C$-modules. Set
$\sJ=\Ann\fie\subset\Oh_C$, and write $B\subset C$ for the subscheme
defined by $\sJ$. Then $B$ is Cohen--Macaulay and $\fie$ has a canonical
factorization of the form
 \begin{equation}
\sF\onto\sF_{|B}\into\om_B=\sHom_{\Oh_C}(\Oh_B,\om_C)\subset\om_C,
 \label{eq:adj}
\nonumber
 \end{equation}
where $\sF_{|B}\into\om_B$ is generically onto.
 \end{PROP}

  \begin{PROP}[ \cite{CF}]\label{h1=0}Let $C$ be a
	 pure  1-dimensional
 projective scheme, let $\sF$ be a rank 1 torsion free sheaf on $C$.

  \begin{enumerate}
	 \renewcommand\labelenumi{(\roman{enumi})}
	
 \item
If
 $ \deg (\sF)_{|B} \geq  2 p_a(B) - 1$
 for every subcurve $B \subseteq C$ then $H^1(C, \sF)=0$.

  \item
If $\sF$ is invertible and
 $ \deg (\sF)_{|B} \geq  2 p_a(B) $
 for every subcurve $B \subseteq C$ then $|\sF|$ is base point free.

   \item
If $\sF$ is invertible and
 $ \deg (\sF)_{|B} \geq  2 p_a(B) +1$
 for every subcurve $B \subseteq C$ then $\sF$ is very ample on $C$.

  \end{enumerate}

 \end{PROP}

 As we mentioned in the Introduction, our approach to the analysis of the ring $R(C, L)=\bigoplus_{k\geq 0} H^0(C,L^{\otimes k})$ for a line bundle $L$ builds on the generalization of Clifford's Theorem proved by the authors in \cite{cliff}. In the rest of this section we recall the  main results we need from \cite{cliff}, namely,  the notion of subcanonical cluster and  Clifford's Theorem,
and we prove some technical lemmas on the cohomology of rank one torsion free sheaves.

\begin{DEF} \label{canonical cluster}
 A {\em cluster} $S$ of {\em degree}
$r$ is a $0$-dimensional subscheme of $C$ with
$\length\Oh_S=\dim_k\Oh_S=r$. A cluster $S \subset C$ is {\em subcanonical} if the space $H^0(C, \sI_S \omega_C)$ contains a generically invertible section, i.e., a section $s_0$ which does not vanish on any subcurve of $C$.
 \end{DEF}

 \begin{TEO}[\cite{cliff}, Theorem A]\label{clifford} Let $C$ be a projective  2-connected curve either reduced with planar singularities  or contained in a smooth  algebraic surface, and let $S \subset C$ be a  subcanonical cluster.

Assume that $S$ is a Cartier divisor or alternatively that there exists
 a generically invertible section $H \in H^0(C, \sI_S K_C)$ such that $\div(H) \cap \Sing(C_{\red})=\emptyset$.

 Then
 $$h^0(C,\sI_S K_C) \leq p_a(C) - \frac{1}{2} \deg (S).$$
 Moreover if equality holds then the pair $(S, C)$ satisfies one of the following assumptions:

 \begin{enumerate}
  \renewcommand\labelenumi{(\roman{enumi})}
\item $S= 0,\, {K_C}$;
\item  $C$ is honestly hyperelliptic and $S$ is a multiple of the honest $g_{2}^{1}$;
\item $C$ is 3-disconnected (i.e.,  there is
a decomposition $C=A+B$ with  \mbox{$A\cdot B=2$)}. \\
\end{enumerate}
 \end{TEO}

 \begin{REM}\label{cliffrem} Let $C$ and $S$ be as in Theorem  \ref{clifford}. Then Riemann-Roch implies that
 $$h^0(C,\sI_S K_C) + h^1(C,\sI_S K_C) \leq  p_a(C) +1$$
 and equality holds if one of the three cases listed in Theorem  \ref{clifford} is satisfied.
 \end{REM}

 \begin{REM} If $Z_0 \subset Z$ are clusters, then the natural restriction map $\Oh_Z \onto \Oh_{Z_0}$ induces an inclusion $\Ext^1(\Oh_{Z_0}, \Oh_C) \into \Ext^1(\Oh_{Z}, \Oh_C)$.
 \end{REM}

  \begin{LEM}\label{estensione} Let $C$ be a Gorenstein curve and $Z$ a cluster. Assume that   there exists an extension  $\xi \in \Ext^1(\Oh_{Z}, \Oh_C)$ such that $\xi \notin \Ext^1(\Oh_{Z_0}, \Oh_C)$ for every proper subcluster $Z_0 \subsetneq Z$. Then the corresponding extension of sheaves can be written as
$$ 0 \to \Oh_C \to \sHom (\sI_Z, \Oh_C) \to \Oh_Z \to 0.$$
 \end{LEM}
 \begin{proof}
 Consider an extension corresponding to $\xi$
\begin{equation}\label{E_xi}0 \to \Oh_C \to E_{\xi}\to \Oh_Z \to 0.\end{equation}

 We prove first that $E_{\xi}$ is torsion free. Indeed, if $\Tor (E_{\xi}) \neq 0$ 
  then there exists a subcluster $Z_0 \subset Z$ and a sheaf $E_0 \cong E_{\xi} / \Tor(E_{\xi})$ which fits in the following commutative diagram:

 $$\xymatrix{& & \Tor(E_{\xi})\ar[r]^{\iso} \ar[d] & \Tor(E_{\xi}) \ar[d] \\
 0 \ar[r]&\Oh_C \ar[r] \ar[d]& E_{\xi}\ar[r]\ar[d] & \Oh_Z\ar[d] \ar[r] & 0\\
0 \ar[r]& \Oh_C \ar[r]\ar[d] & E_{0}\ar[r]\ar[d] & \Oh_{Z_0}\ar[d] \ar[r] & 0\\
& 0 & 0 & 0&}$$
In particular there exists a proper subcluster $Z_0$ such that the extension corresponding to  $E_0$ in $\Ext^1(\Oh_{Z_0}, \Oh_C)$ corresponds to $\xi$, which is impossible.\\

Since $E_{\xi}$ is a rank 1 torsion free sheaf it is reflexive, i.e., there is a natural isomorphism $\sHom(\sHom(E_{\xi}, \Oh_C), \Oh_C) \cong E_{\xi}$. Dualizing sequence (\ref{E_xi}) we see that $\sHom(E_{\xi}, \Oh_C) \cong \sI_Z$ since
$\sExt^1(\Oh_Z , \Oh_C)\iso \Oh_Z$ and  $\sExt^1( E_{\xi} , \Oh_C)=0$, 
  hence $E_{\xi} \cong  \sHom (\sI_Z, \Oh_C) $.

\end{proof}

\begin{REM}  Throughout the paper we repeatedly use the natural isomorphisms  $\Oh_S \iso \Oh_S \cdot L$  and
$  H^0(S, \Oh_S)^{\ast} \iso  \Ext^1(\Oh_{S}, \omega_C \otimes L^{-1}) \iso  \Ext^1(\Oh_{S} \cdot L , \omega_C \otimes L^{-1})  $
for every cluster $S$ on $C$ and for every invertible sheaf $L$.

\end{REM}

A useful tool in the analysis of the multiplication map $ H^0(C, L)^{\otimes 2} \to H^0(C, L^{\otimes 2})$ is the restriction to a suitable cluster $S$.
 Indeed the composition of the multiplication map
$ H^0(C, L)^{\otimes 2} \to H^0(C, L^{\otimes 2})$
with the   evaluation map
$H^0(C, L^{\otimes 2}) \to H^0(S, \Oh_S) $
yields a natural map
$H^0(C, L) \otimes H^0(C,L) \to H^0(S, \Oh_S).$

\begin{LEM}\label{non surj} Let $C$ be a Gorenstein curve and $L$ an effective line bundle on $C$. Let $S$  be a cluster such that the restriction map
$$H^0(C, L) \otimes H^0(C,L) \to H^0(S, \Oh_S)$$
is not surjective. Then there exists a nonempty  subcluster $S_0 \subseteq S$ such that
$$h^0(C, L) + h^1(C, L) \leq h^0(C, \sI_{S_0} L) + h^1(C, \sI_{S_0} L).$$
\end{LEM}
\begin{proof} Let  $S$  be a cluster such that the restriction map $$H^0(C, L) \otimes H^0(C,L) \to H^0(S, \Oh_S)$$ is not surjective.  By Serre duality  the dual map
$$\varphi \colon  \Ext^1(\Oh_S, \omega_C \otimes L^{-1}) \to \Hom(H^0(C, L), H^1(C, \omega_C \otimes L^{-1}))$$
is not injective. The dual map $\varphi$ is given as follows: consider an element $\xi \in \Ext^1(\Oh_S, \omega_C \otimes L^{-1})$ and its corresponding extension

$$ 0 \to \omega_C \otimes L^{-1} \to E_{\xi} {\longrightarrow}  \Oh_S \to 0.$$
Let $c_{\xi} \colon H^0(S, \Oh_S)\to  H^1(C, \omega_C \otimes L^{-1}) $ be the connecting homomorphism induced by the extension. Then
the restriction map $r \colon  H^0(C, L) \to H^0(S, \Oh_S)$ induces a map $\varphi_{\xi} = c_{\xi} \circ r  \colon H^0(C, L) \to H^1(C, \omega_C \otimes L^{-1}))$ given as follows
$$\xymatrix{& && H^0(C, L) \ar[d]^r \ar[dr]^{\varphi_{\xi}}&\\
0 \ar[r]& H^0(C,\omega_C \otimes L^{-1}) \ar[r]& H^0(C,E_{\xi}) \ar[r]^{f_{\xi}}& H^0(S, \Oh_S)\ar[r]^{c_{\xi} \ \ \ \ \ \ \ } & H^1(C, \omega_C \otimes L^{-1})}$$
The map $\varphi_{\xi}$ is precisely $\varphi(\xi) \in \Hom(H^0(C, L), H^1(C, \omega_C \otimes L^{-1}))$. By definition $\varphi(\xi)=0$ if and only if $\operatorname{Im}(r) \subset \operatorname{Im} (f_{\xi})$. In particular if $\varphi(\xi)=0$ then we have $\dim \operatorname{Im}(r) \leq \dim \operatorname{Im} (f_{\xi})$ which implies that
\begin{equation}\label{inuguaglianza Clifford} h^0(C, L) + h^1(C, L) \leq h^0(C, E_{\xi}) +  h^0(C, \sI_S L).\end{equation}

In order to prove the Lemma let  ${S_0}$ be  minimal (with respect to the inclusion) among the subclusters of $S$ for which the restriction  $\Sym^2 H^0(C, L) \to H^0(S_0, \Oh_{S_0})$ fails to be surjective. This implies that if $Z \subsetneq S_0$ is any proper subcluster then the map $$\varphi_0 \colon \Ext^1(\Oh_{Z}, \omega_C \otimes L^{-1})\to\Hom(H^0(C,L), H^1(C, \omega_C \otimes L^{-1})) $$ is injective. Note that $\varphi_0$ factors through $\varphi$:
$$\xymatrix{\Ext^1(\Oh_{Z}, \omega_C \otimes L^{-1}) \ar[d]\ar[dr]^{\varphi_0}\\
\Ext^1(\Oh_{S_0}, \omega_C \otimes L^{-1})\ar[r]^-{\varphi} &  \Hom(H^0(C, L), H^1(C, \omega_C \otimes L^{-1}))}$$

By the minimality of $S_0$ if $\xi \in \Ext^1(\Oh_{S_0}, \omega_C \otimes L^{-1})$ is in  the kernel of $\varphi$, it must not belong to the image of $\Ext^1(\Oh_{Z}, \omega_C \otimes L^{-1})$  for every $Z \subsetneq S_0$. The corresponding extension $E_{\xi}$ is isomorphic to $\sHom (\sI_{S_0} , \Oh_C) \otimes \omega_C \otimes L^{-1}\iso \sHom (\sI_{S_0}L , \omega_C) $ thanks to Lemma \ref{estensione}. Thus $h^0(C, E_{\xi})=h^1(C, \sI_{S_0} L)$ by Serre duality. Inequality (\ref{inuguaglianza Clifford}) becomes
$$h^0(C, L) + h^1(C, L) \leq h^0(C, \sI_{S_0} L) + h^1(C, \sI_{S_0} L).$$\end{proof}

 \section{Noether's Theorem for  singular curves}
 The aim of this section is to prove  Noether's Theorem for singular curves. For the proof we use two main ingredients: a generalization of the \emph{free pencil trick} (see Lemma \ref{free pencil trick}), and the surjectivity of the restriction map $H^0(C, \omega_C)^{\otimes 2} \to H^0(S, \Oh_S)$ for a suitable cluster $S$.

\begin{LEM}\label{esiste free pencil} Let $C$ be a projective curve  which is either  reduced with planar singularities   or contained in a smooth  algebraic surface. Assume that $C$ is 2-connected and $p_a (C) \geq 2$.  Let  $H \in H^0(C,\omega_C)$ be a generic section.

Then there exists a cluster $S$ contained in $\operatorname{div} H$ such that the following hold:
\begin{enumerate}
\item $h^0(C, \sI_S \omega_C)=2$
\item the evaluation map $H^0(C, \sI_{S}K_C)\otimes \mathcal{O}_C \to \sI_{S}\omega_C$ is surjective.
\end{enumerate} \end{LEM}
\begin{proof} Since $|K_C|$ is base point free thanks to Theorem \ref{thm:curve} we may assume that $H$ is generically invertible and $\div H$ is a length $2 p_a(C)-2$ cluster. Thus
 for every integer $\nu \in \{1,\ldots, p_a(C)\}$ there exists at least one cluster $S_{\nu} \subseteq \div H$ such that $h^0(C, \sI_{S_{\nu}} \omega_C)=\nu$. In particular we may take a cluster $S$ such that \mbox{$h^0(C, \sI_S \omega_C)=2$} and $S$ is maximal up to inclusion among the clusters contained in $\operatorname{div} H$ with this property. $S$ is the desired cluster. Indeed, if it were  $S_0 \supsetneq  S$ such that the image of the  evaluation map $H^0(C, \sI_{S}K_C)\otimes \mathcal{O}_C \to \sI_{S}\omega_C$  was $\sI_{S_0} \omega_C \subsetneq \sI_{S}\omega_C$  then we would have  $h^0(C, \sI_{S_0} K_C)=2$, contradicting  the maximality of $S$.\end{proof}\\

Even though the sheaf $\sI_S \omega_C$ defined in the above Lemma is not usually a line bundle by abuse of notation we will call it a \emph{free pencil}.

\begin{LEM}\label{free pencil trick} Let the pair $(C, S)$ be as in the previous lemma. Then the map \begin{equation}\label{surj pencil} H^0(C, \sI_S \omega_C) \otimes H^0(C, \omega_C) \stackrel{m}{\longrightarrow} H^0(C, \sI_S \omega_C^{\otimes 2})\end{equation} is surjective. \end{LEM}
 \begin{proof}
  Consider the evaluation map $H^0(C,\mathcal{I}_{S}\omega_C) \otimes \omega_C \stackrel{ev}{\to} \mathcal{I}_{S}\omega_C^{\otimes 2}$ and its kernel~$\sK$:
\begin{equation}\label{seqence commutators}
0\to \mathcal{K}\to H^0(C,\mathcal{I}_{S}\omega_C) \otimes \omega_C \to \mathcal{I}_{S}\omega_C^{\otimes 2} \to 0.
\end{equation}
The map  (\ref{surj pencil}) is surjective if and only if $h^1(C,\mathcal{K})=2 $ since
$h^1(C, \mathcal{I}_{S}\omega_C^{\otimes 2})=0$ by Proposition \ref{h1=0}. In the rest of the proof we establish  $h^1(C,\mathcal{K})=2 $.\\

We have  $$\sK \cong \sHom (\sI_S \omega_C, \omega_C).$$ Indeed consider a basis $\{x_0, x_1\}$ for $H^0(C,\mathcal{I}_{S}\omega_C)$ and define the map
\begin{eqnarray*}
\iota \colon  \sHom(\mathcal{I}_{S}\omega_C, \omega_C) &\to& H^0(C,\mathcal{I}_{S}\omega_C) \otimes \omega_C\\
\varphi & \mapsto & x_0 \otimes \varphi( x_1) - x_1 \otimes \varphi( x_0).
\end{eqnarray*}
Our aim is to check that $\iota$ is injective and $\iota(\sHom(\mathcal{I}_{S}\omega_C, \omega_C))$ is precisely $\sK$. It is clear that $ \im(\iota) \subset \sK$. Moreover $\iota$  is injective since the sheaf $\sI_S \omega_C$ is generated by its sections $x_0$ and $x_1$. It is straightforward to check that over the points $P \in C$ not belonging to $S$ (where both the sheaves $\sHom(\mathcal{I}_{S}\omega_C, \omega_C)$ and $\sK$ are invertible), $\iota$ induces an isomorphism. Moreover  computing the Euler characteristic  we have
$$\chi(\sHom(\mathcal{I}_{S}\omega_C, \omega_C))= \chi (\sK) = \deg S - (p_a(C)-1) $$
hence the map $\iota$ induces an isomorphism between $\sHom(\mathcal{I}_{S}\omega_C, \omega_C)$ and $\sK$.

We know that $$H^1(C,\mathcal{K})^{\ast}= \Hom (\mathcal{K}, \omega_C)=H^0(C,\sHom(\mathcal{K}, \omega_C) )$$
It is easy to check that $\sI_S \omega_C$ is reflexive, i.e., $$\sHom(\sHom(\sI_S \omega_C, \omega_C), \omega_C) \cong \sI_S \omega_C$$ thus $h^1(C,\mathcal{K})= h^0(C, \sI_S \omega_C)=2$.
 \end{proof}\\

We may now prove our main theorem.

\begin{TEO}\label{noether}  Let $C$ be a  projective curve either reduced with planar singularities or contained in a smooth  algebraic surface. Assume that $C$ is 3-connected, not honestly hyperelliptic and $p_a (C) \geq 3$. Then the map
$$\Sym^n H^0(C, \omega_C) \to H^0(C, \omega_C^{\otimes n})$$
is surjective for every $n \geq 0$.
\end{TEO}

\begin{proof} It is already known that the canonical ring $R(C,\omega_C)= \bigoplus_{n \geq 0}H^0(C, \omega_C^{\otimes n})$ is generated in degree at most 2: see Konno \cite[Prop. 1.3.3]{Konno123} or  Franciosi \cite[Th. C]{fr_adj}. Notice that, even though both papers deal with the case of divisors on smooth surfaces, their proofs go through without changes to reduced Gorenstein curves. Thus to prove the theorem  it is sufficient to show  that the map in degree 2 is surjective:

$$H^0(C, \omega_C) \otimes H^0(C, \omega_C) \to H^0(C, \omega_C^{\otimes 2}).$$

We consider a free pencil $\sI_S \omega_C$ (as in Lemma \ref{esiste free pencil}) and study the following commutative diagram:

$$\xymatrix{H^0(C, \sI_S \omega_C) \otimes H^0(C, \omega_C) \ar@{^{(}->}[r] \ar[d]^{m} & H^0(C, \omega_C)^{\otimes 2} \ar[r] \ar[d]^r \ar[dr]^{p} & H^0(S, \Oh_S) \otimes H^0(C, \omega_C) \ar[d]\\
H^0(C, \sI_S \omega_C^{\otimes 2}) \ar@{^{(}->}[r]  & H^0(C, \omega_C^{\otimes 2}) \ar[r] & H^0(S, \Oh_S)
}$$
A simple diagram chase shows that if both the maps $m$ and $p$ are surjective, then the product map $r$ is surjective too, proving the theorem.\\

 Lemma \ref{free pencil trick} states precisely that the map $m$ is surjective.

The map $p$ must be surjective too: if not, we could apply Lemma \ref{non surj} and conclude that there exists a nonempty subcanonical cluster $S_0 \subseteq S$, contained in a generic section in $H^0(C, \omega_C)$, such that
$$h^0(C, \sI_{S_0} \omega_C)+ h^1(C, \sI_{S_0} \omega_C) \geq p_a(C)+1.$$
By Theorem \ref{clifford} and Remark \ref{cliffrem} we know that this can not happen if $C$ is 3-con\-nected and not honestly hyperelliptic and $S_0\neq \emptyset, K_C$. \end{proof}\\

Theorem \ref{noether} does not hold for every curve with very ample canonical sheaf, but only for the 3-connected ones. Indeed, Noether's Theorem may be false for canonical, non reduced and 3-disconnected curves, as shown in the following example.

\begin{EX}\label{esempio} Let $B$ be a smooth genus $b$ curve with $b \geq 4$ and let $D$ be a general effective divisor on $B$ of degree $b+3$. The linear system $|D|$ is very ample and induces an embedding of $C$ in $\p^3$ (see \cite[Ex. V.B.1]{ACGH}).

Define the ruled surface $X = \p_B (\Oh_B \oplus \Oh_B(D-K_B))$: the map $f \colon  X \to B$ has a section $\Gamma$ with selfintersection $(-b+5)$ (see \cite[\S V.2]{hart} for the main numerical properties). Consider the curve $C= 2 \Gamma$: we have that $p_a(C)= b+4$ and $C$ is 2-disconnected (numerically disconnected if $b\geq 5$). By adjunction we have
$$K_C = (K_X + C)_{|C}=f^{\ast}(D)$$
and it is easy to check that it is very ample on $C$ by analyzing the standard decomposition
$$0 \to \omega_{\Gamma} \to \omega_{C} \to \omega_{C|\Gamma} \to 0. $$
Since $\Gamma$ is a section of $f \colon  X \to B$, $F$ induces an isomorphism $(B, \Oh_B(D)) \iso (\Gamma, \omega_{C|\Gamma} )$.
Therefore it is immediately seen that $|\omega_C|$ separates length 2 clusters.

The map
$$\Sym^2H^0(C, \omega_C) \stackrel{q_0}{\longrightarrow}  H^0(C, \omega_C^{\otimes 2})$$
is not surjective, as one could see from the following diagram:

$$\xymatrix{ H^0(\Gamma, \omega_{\Gamma}) \otimes H^0(C, \omega_C) \ar[r] \ar[d] & H^0(C, \omega_C)^{\otimes 2}  \ar[r] \ar[d]^q & H^0(\Gamma, \omega_{C|\Gamma}) \otimes H^0(\Gamma, \omega_{C|\Gamma}) \ar[d]^p\\
 H^0(\Gamma, \omega_{\Gamma} \otimes \omega_C) \ar[r] &  H^0(C, \omega_C^{\otimes 2})  \ar@{->>}[r]& H^0(\Gamma, \omega_{C|\Gamma}^{\otimes 2}) }$$
Indeed if the map $q_0$ was surjective (hence $q$), then the map $p$ would be surjective as well. Since $\omega_{C|\Gamma} \iso \Oh_B(D)$ the image of the map $p$ is the same as the image of the map $$p_0 \colon \Sym^2 H^0(B, \Oh_B(D)) \to H^0(B, \Oh_B(2D))$$ which is not surjective, as one can easily check by computing the dimension of the two spaces.

\end{EX}

\section{Castelnuovo's Theorem for reduced curves}
In this section we  prove 
a generalization of Castelnuovo's Theorem for reduced curves.

In the proof we will apply Lemma \ref{non surj}  and the following Proposition, which is a Clifford-type result for line bundles of high degree.

\begin{PROP}\label{cliff_grosso} Let $C$ be a projective reduced  curve with planar singularities  and $L$ a line bundle on $C$ such that
$$ \deg L_{|B} \geq 2 p_a (B) +1 \quad \text{for every } B \subset C.$$

If $S$ is a cluster contained in a generic section $H \in H^0(C, L)$ then
$$h^0(C, \sI_S L) + h^1(C, \sI_S L) < h^0(C,L). $$
\end{PROP}

\begin{proof}
Notice at first that $H^1(C, L)= 0$ and $|L|$ is very ample by Proposition \ref{h1=0}.
Therefore a generic hyperplane section consists  of $\deg L$ smooth points.
 Moreover by
 Riemann Roch Theorem we have
$$\begin{array}{rll}h^0(C, \sI_S L) + h^1(C, \sI_S L) < h^0(C,L) &  \Longleftrightarrow &
h^0(C, \sI_S L) < h^0(C, L) - \frac12{\deg S} \\ &   \Longleftrightarrow &
h^1(C, \sI_S L) <  \frac12{\deg S}\end{array}$$

We argue by induction on the number of irreducible components of $C$. Suppose that $C$ is irreducible or that the statement holds for every reduced curve with fewer components of $C$. If $C$ is disconnected the statement is trivial, hence we may assume that $C$ is connected. If  $h^0(C, \sI_S L)=0$ or $h^1(C, \sI_S L)=0$ the result is trivial too, thus we may assume $h^0(C, \sI_S L) > 0$ and $h^1(C, \sI_S L) > 0$.\\

Suppose at first that there exists a proper subcurve $B \subset C$ such that $$H^1(C, \sI_S L) \iso H^1(B,  \sI_S L_{|B}).$$
By induction we have that $h^1(C, \sI_S L) <  \frac12 {\deg S_{|B}} \leq  \frac12 {\deg S}$ and we may conclude.\\

If there is no subcurve $B \subset C$ as above (e.g. when $C$ is irreducible) we can easily deduce from Proposition \ref{lem:adj} that there exists a generically surjective map $\sI_S L \into \omega_C$, hence we may assume that there exists a subcanonical cluster $Z$ such that
$$\sI_S L \iso \sI_Z \omega_C.$$
Since $S$ is contained in a generic section of $H^0(C, L)$ it is a Cartier divisor, hence $Z$ is a Cartier divisor too.

If $C$ is 2-connected we apply Theorem \ref{clifford} and we conclude since
\begin{eqnarray*} h^0(C, \sI_C L) &=& h^0(C, \sI_Z \omega_C) \leq p_a(C) - \frac12 \deg Z \\ &=&\frac12 \deg L+1 - \frac12 \deg S < h^0(C, L) - \frac12 \deg S.\end{eqnarray*}

If $C$ is 2-disconnected, we can find a decomposition $C= C_1 + C_2$, such that $C_1\cdot C_2=1$ and $C_1$ is 2-connected (see \cite[Lemma A.4]{CFM}). Thus we consider the following exact sequence:
$$0\to
 {\sI_Z}_{|C_1} \omega_{C_1} \to \sI_Z \omega_C \to (\sI_Z \omega_C)_{|C_2} \to 0.$$

We know by induction that
\begin{eqnarray*} h^0(C_2,(\sI_Z \omega_C)_{|C_2})&=& h^0(C, (\sI_S L)_{|C_2}) < h^0(C_2, L_{|C_2})-\frac12  \deg S_{|C_2}\\&=& -p_a(C_2)+1+\deg L_{|C_2}-\frac12 \deg S_{|C_2}.   \end{eqnarray*}
We apply Theorem \ref{clifford} to ${\sI_Z}_{|C_1} \omega_{C_1}$ since the single point $C_1 \cap C_2$ is a base point for $|\omega_C|$, hence the space $H^0(C_1,  {\sI_Z}_{|C_1}\omega_{C_1})  \cong H^0(C_1,  {\sI_Z}_{|C_1}\omega_{C|C_1}) $ contains an invertible section, that is  $Z_{|C_1}$ is a subcanonical cluster.  Thus
$$h^0(C_1,  {\sI_Z}_{|C_1} \omega_{C_1}) \leq p_a(C_1)- \frac12 \deg Z_{|C_1}=\frac12 \deg L_{|C_1}-\frac12 \deg S_{|C_1} +\frac12.$$
and we conclude since
\begin{eqnarray*}h^0(C, \sI_S L)&=& h^0(C, \sI_Z \omega_C) \leq  h^0(C_2,(\sI_Z \omega_C)_{|C_2}) + h^0(C_1,  {\sI_Z}_{|C_1} \omega_{C_1})\\
&< & -p_a(C_2)+1+\deg L_{|C_2}-\frac12  \deg S_{|C_2} + \frac12 \deg L_{|C_1}-\frac12 \deg S_{|C_1} +\frac12\\
&\leq & h^0(C, L) - \frac12 \deg S.\end{eqnarray*}

\end{proof}

\begin{TEO}\label{castelnuovo} Let $C$ be a  projective reduced  curve with planar singularities  and let $L$ be a line bundle on $C$ such that
$$ \deg L_{|B} \geq 2 p_a (B) +1 \quad \text{for every } B \subset C.$$

Then the product map
$$\Sym^n H^0(C, L) \to H^0(C, L^{\otimes n})$$
is surjective for every $n \geq 1$.
\end{TEO}

\begin{proof}
Notice at first that $H^1(C, L)= 0$ and  $L$ is very ample  by Proposition \ref{h1=0}. \\

If $n \geq 2$ the map
$$H^0(C, L^{\otimes n}) \otimes H^0(C, L) \to H^0(C, L^{\otimes (n +1)})$$
is surjective by  \cite[Prop. 1.5]{fr_adj}  since $H^1(C, L^{\otimes n} \otimes L^{-1})=0$.  In order to prove the theorem we check that the map in degree 2 is surjective:
$$\Sym^2 H^0(C, L) \to H^0(C, L^{\otimes 2}).$$
To this aim we consider a generic hyperplane section $S=\div L$ and the following commutative diagram
$$\xymatrix{H^0(C, \sI_S L) \otimes H^0(C, L) \ar@{^{(}->}[r] \ar@{->>}[d] & H^0(C, L) \otimes H^0(C, L) \ar[r] \ar[d]^r \ar[dr]^{p} & H^0(S, \Oh_S) \otimes H^0(C, L) \ar[d]\\
H^0(C, \sI_S L^{\otimes 2}) \ar@{^{(}->}[r]  & H^0(C, L^{\otimes 2}) \ar@{->>}[r] & H^0(S, \Oh_S)
}$$
Notice that the first column is surjective since $\sI_S L \iso \Oh_C$
while the second row is exact since $H^1(C, \sI_S L^{\otimes 2}) \iso H^1(C, L)=0$. A simple diagram chase shows that the map $r$ is surjective if and only if the map $p$ is surjective.

It is $h^0(C, \sI_{S_0} L) + h^1(C, \sI_{S_0} L) < h^0(C,L)$  for every subcluster $S_0\subseteq S$
 by Proposition \ref{cliff_grosso}, hence the map $p$ must be surjective
by Lemma \ref{non surj}.
\end{proof}

\begin{REM}
If $C$ is numerically connected our result implies that the embedded curve $\varphi_L (C) \subset \proj H^0(C,L)^{\ast}$ is arithmetically Cohen-Macaulay.
\end{REM}

\vskip.5cm
Marco Franciosi\\
Dipartimento di Matematica, Universit\`a di Pisa\\
Largo B.Pontecorvo 5,  I-56127 Pisa (Italy)\\
{\tt franciosi@dm.unipi.it}
\vspace{.5cm}\\
Elisa Tenni\\
SISSA International School for Advanced Studies\\
via Bonomea 265, I-34136 Trieste (Italy)\\
{\tt etenni@sissa.it}

\end{document}